\documentclass[11pt]{article}
\usepackage{amsmath,amssymb,amsfonts}
\usepackage{amsthm,todonotes, txfonts}
\usepackage{rotating} 
\usepackage[normalem]{ulem}

\usepackage{tikz-cd}

\usepackage{color}

\newcommand{\editA}[1]{{\color{black}#1}} 

\usepackage{fullpage}

\usepackage{tikz-cd, tikz} 
\usetikzlibrary{decorations.pathmorphing}
\usepackage{transparent}
\usepackage[alphabetic]{amsrefs}

\usepackage{color}

\usepackage{fullpage}

\newtheorem{theorem}{Theorem}[section]
\newtheorem{claim}[theorem]{Claim}
\newtheorem{definition}[theorem]{Definition}

\newtheorem{proposition}[theorem]{Proposition}
\newtheorem{lemma}[theorem]{Lemma}

\DeclareMathOperator{\Aff}{Aff}
\DeclareMathOperator{\Sub}{Sub}
\DeclareMathOperator{\Heis}{Heis}
\DeclareMathOperator{\GL}{GL}
\DeclareMathOperator{\SL}{SL}

\DeclareMathOperator{\PGL}{PGL}
\DeclareMathOperator{\SO}{SO}

\DeclareMathOperator{\Ad}{Ad}

\newcommand{\RR}{\mathbb{R}}

\newcommand{\ZZ}{\mathbb{Z}}

\newcommand{\CC}{\mathbb{C}}
\newcommand{\fraksl}{\mathfrak{sl}}
\newcommand{\frakgl}{\mathfrak{gl}}
\newcommand{\frakb}{\mathfrak{b}}

\definecolor{green}{rgb}{0.0, 0.5, 0.0}
\definecolor{purple}{rgb}{0.5, 0.0, 0.5}
\definecolor{bluegreen}{rgb}{0.0,0.5, 0.5}
\definecolor{orange}{rgb}{1,0.5, 0.1}
\definecolor{redgreen}{rgb}{0.5, 0.5, 0.0}

\def\green{\color{green}}

\def\green{\color{green}}

\def\g2{{\green 2}}

\title{Local Limits of Connected Subgroups of $\SL_3(\mathbb{R})$}
\author{Nir Lazarovich and Arielle Leitner} \date{}
\begin{document}
\maketitle

\begin{abstract}
In this paper we describe the local limits under conjugation of all closed connected subgroups of $\SL_3 ( \RR)$ in the Chabauty topology. 
\end{abstract} 

\section{Introduction}%%%%%%%%%%%%%%%%%%%%%%%%%%%%%%

Let $G$ be a locally compact second countable group.  The set $\Sub(G)$ of closed subgroups of $G$ may be endowed with the \emph{Chabauty topology}, with which it is a compact space.  The Chabauty topology was introduced in several places, among them \cite{chab, Fell}.   For an overview we recommend \cite{harpe}. We are interested in \emph{convergence} in the Chabauty topology, which for a locally compact second countable group can be defined as follows: 

\begin{definition}[See \cite{CEG}]\label{conv}  A sequence of closed subgroups $\{ H_n \} \leq G$ {\emph converges} to $H \leq G$ if the following two conditions hold
\begin{enumerate} 
\item For every $h \in H$, there exists a sequence of elements $h_n \in H_n$ so that $h_n \to h$. 
\item Given a sequence of elements $h_n \in H_n$, for every convergent subsequence $({h_n}_k) \to h$ then $h \in H$. 
\end{enumerate} 
A group $H \leq G$ \emph{converges to a group} $L\leq G$ \emph{ under conjugacy} if there exists a sequence $p_n \in G$ such that $p_n H p_n^{-1}$ converges to $L$ in the sense of the definition above. 

A connected subgroup $H$ \emph{locally converges} to a connected subgroup $L$ under conjugacy if there is a sequence $p_n \in G$ such that $ p_n H p_n ^{-1} \to L'$ and $L$ is the identity component of $L'$.  
\end{definition}

 %See \cite{CDW} sections 3.1 and 3,2, and in particular example 2.  The group of rotations can limit to a parabolic semidirect product with the diagonal group that has $\pm 1$ on the diagonal.   Here we define limits to be connected components so that we can go back and forth between groups and algebras.}{\blue: N: I find it too vague. How about we write: 
Limits of connected subgroups are not necessarily connected. For instance, Example 2 in  \cite[Section 3.2]{CDW} shows that there is a sequence of conjugates of the rotation subgroup $\SO(2)\le \SL_2(\RR)$ that converges to the subgroup %$\{\pm u \; |\; u\text{ is upper triangular unipotent}\}$, 
$ \left  \{ \left ( \begin{smallmatrix} \pm 1 & x \\0 &\pm1 \end{smallmatrix} \right ): x \in \RR \right \} $%\cup \left \{\left (\begin{smallmatrix} -1 & x \\0 &-1 \end{smallmatrix}\right ) : x \in \RR \right \},  $ 
which has two connected components. It is therefore necessary to pass to the identity component of the limit in the definition of local convergence.

%\sout{ Local convergence does not give a subspace of $\Sub(G)$ nor a quotient.} 
%One can understand the \emph{Chabauty compactification}, $\Sub(G)$ by understanding the space of conjugates of each subgroup, as a cell of dimension $G/ N_G(H)$, and then how these cells are glued together by which groups limit to others.   

Let $G=\SL_3(\RR)$.  Our main result is a description of the local convergence of connected subgroups of $G$. Note that we consider subgroups up to conjugacy, and use the classification of subalgebras of $\mathfrak{g}$ up to conjugacy by Winternitz \cite{Winternitz}, their images under the exponential map are connected subgroups of $G$.   In Section \ref{classification}, we have written a subsection for each dimension of subgroups. 
	
	\begin{theorem}\label{main}
		The local convergence of the connected subgroups of $\SL_3(\RR)$ of each dimension is described by the chart of limits in the corresponding section in the paper.
	\end{theorem}

The connected subgroups of $\SL_3 (\RR)$ were classified by Winternitz \cite{Winternitz}, who provided a full list of subalgebras of $\mathfrak{sl}_3(\RR)$ up to conjugacy. %Here we have applied the exponential map.  
 In each section, %\sout{W}
we first list the subgroups %\sout{with his notation} 
together with their normalizers and properties, and then we %\sout{give} 
provide a chart which shows which groups locally limit to others by conjugation. Following each chart, we prove that this is indeed the complete chart of limits. %\sout{, and describes how to construct each part by gluing together orbits.}

	Theorem \ref{main} gives a partial understanding of the closure of the connected subgroups in $\Sub(G)$ in the following sense.
%\sout{The group $G$ acts on $\Sub(G)$ by conjugation. The orbit of each} 
The conjugacy class of each connected subgroup $H \in \Sub(G)$ is %\sout{cell of dimension} 
a subspace homeomorphic to $G/ N_G(H)$. 
%\sout{One discovers how orbits are glued together by understanding limits of $H$.} 
The closure of the conjugacy class of $H$ in $\Sub(G)$ consists of conjugacy classes of subgroups whose identity components are the local limits of $H$ described by our main result.

 %\sout{ In general, understanding the topology of $\Sub(G)$ is a difficult question, and the homeomorphism type of $\Sub(G)$ is an open problem for all but a few groups.  Let $G= \SL_3(\RR)$.  It would be difficult to determine all of $\Sub(G)$, so in this paper we consider limits only of connected subgroups.  
%We make progress in determining the homeomorphism type of the \emph{subspace of connected subgroups} denoted $\Sub_c(G)$. }

Using work of \cite{CDW, Dries} we prove the following proposition which is a component of the proof of Theorem \ref{main}.

\begin{proposition}\label{samedim} 
	Let $G=\SL_3(\RR)$ and let $H \leq G$ be a connected subgroup and $L$ a local limit of $H$.  Then $\dim H = \dim L$. 
\end{proposition}

Note that this fails in $\SL_4(\RR)$ as shown in \cite{CDW}. This also implies that the local limit can be seen as a limit of the corresponding Lie subalgebras under the $\Ad(G)$ action. %\todo{change references accordingly to Proposition 0.3 and Theorem 0.2}

	The proof of Proposition \ref{samedim} relies on the following Theorem, which we believe might be of independent interest. %Note that it holds for groups which are not algebraic, so it is more general than \cite{CDW}. 
\begin{theorem}\label{1d limits}
	Let $X\in \frakgl_d(\CC)$ be a matrix with an eigenvalue which is not purely imaginary. Then the local conjugacy limits of the one-dimensional closed subgroup $H=\{\exp(tX) \;|\; t\in \RR \} \le \GL_d(\CC)$ are one-dimensional.
	%Moreover, the same conclusion holds for one dimensional closed subgroups $H=\{\exp(tX) \;|\; t\in \RR \} \le \SL_d(\CC)$ where $X\in \fraksl_d(\CC)$ is a matrix with an eigenvalue which is not purely imaginary. 
\end{theorem}

%\begin{theorem}\label{main}  \sout{Let $G=\SL_3(\RR)$.  }
%\begin{enumerate} 
%\item \sout{Let $H \leq G$ be a connected subgroup and $L$ a limit of $H$.  Then $\dim H = \dim L$. }
%\item \sout{Then $\Sub_c(G)$ has 8 parts, one for each dimension of subgroup.  Each part is constructed by gluing together orbits as described by the chart of limits in each section. }
%\end{enumerate} 
%\end{theorem} 

In all but a few cases the homeomorphism type of all of $\Sub(G)$ is still unknown.  However, progress towards understanding the topology has been made on the Heisenberg group by \cite{BHK}, on $\RR \times \ZZ$ by \cite{Haettel2}, on the set of Cartan subgroups of $\SL_n(\RR)$ by \cite{Leitnersl3, Leitnersln, Haettel}, on $\RR^2$ by \cite{HP}, on $\RR^n$ by \cite{kloeckner}, and \cite{CDW} make progress on limits of symmetric subgroups in $\PGL_n (\RR)$.

%Sometimes a space $X$ can be identified with the a subset of closed subgroups of a locally compact group $G$, i.e., by looking at point stabilizers.  
The Chabauty compactification $\Sub(G)$ may be used to compactify Bruhat-Tits buildings or symmetric spaces, by identifying points in those spaces with their stabilizers in $G$ viewed as points in $\Sub(G)$. %Then the closure of the image of $X$ in $\Sub(G)$ provides a natural compactification of $X$ and gives a useful tool for studying $X$.  
See \cite{GJT, GR, CLV}. 

Aknowldegements: Lazarovich was supported by ISF grant No.1562/19 and by the German-Israeli Foundation for Scientific Research and Development.  Leitner was partially supported by the
ISF-UGC joint research program framework grant 1469/14 and 577/15 and 704/08.
We thank Corina Ciobotaru and Marc Burger for helpful discussions, and the latter for suggesting we look at O-minimal structures.    We thank the referee for many insightful comments which improved the paper. 

\section{Notation and Methods}%%%%%%%%%%%%%%%%%%%%%%%%%%%

As a convention, throughout the paper: $a,b,z$ are parameters for infinite families, $s,t,*$ are variables.  We think of $\CC$ as $( \begin{smallmatrix} * \\ * \end{smallmatrix})$, and $\mathbb{C}^*$ as $ ( \begin{smallmatrix} s & t \\ -t & s \end{smallmatrix} )$. We avoid using set builder notation as it is clunky. 

   Throughout this paper, we will follow the notation given by Winternitz \cite{Winternitz} where $W_{d,k}$ denotes the $k$th group in the list of subgroups of $\SL_3(\RR)$ with Lie algebra of dimension $d$.  Sometimes a group depends on some parameters, which we denote in the superscript; sometimes a group is not conjugate to its transpose, and so we might list them together using the $/$ notation. See e.g. the 3-dimensional groups on p.10.  

We begin by proving Proposition \ref{samedim}. To do so, we first review some theorems which we apply throughout the paper.  

%The first Theorem has two important statements.  First, limits of algebraic groups are algebraic. Second, the dimension of an algebraic group is preserved under taking a limit. 

\begin{theorem}[ \editA{\cite[Theorem 3.1]{CDW}}]\label{algebraic}  Let $G$ be an algebraic group (defined over $\CC$ or $\RR$). Suppose that $H$ is an algebraic
subgroup and $L$ a conjugacy limit of $H$. Then $L$ is algebraic and $\dim L = \dim H$.
\end{theorem}

A more general notion than an algebraic group is a \emph{definable} group, in the sense of an O-minimal structure, see \cite{Coste}.   Many of the properties of algebraic sets carry over to this more general setting. \editA{\cite[Proposition 3.1]{Dries}} implies that the limit of a definable group is definable, and the dimension stays constant under taking a limit. 
%{\color{purple} Before the proof of the next proposition, we motivate with an example of a non-definable group from \cite{CDW} section 3.2, where the dimension of the limit increases.}

There are three flavors of non-algebraic groups among connected subgroups of $\SL_3(\RR)$: 
$$ \begin{pmatrix} e ^{at } & \cdot & \cdot \\
0 & e ^{bt}  & \cdot \\
0 & 0 & e^{-(a+b)t}
\end{pmatrix} : 
a,b \in \mathbb{R} \textrm{ fixed} 
\qquad 
\begin{pmatrix} 
e ^t & t e^t & \cdot \\
0 & e^t & \cdot \\
0 & 0 & e ^{-2t} 
\end{pmatrix} 
\qquad 
\begin{pmatrix}  e^ {zt} & \vdots \\
0 & e ^{-2\mathfrak{R}(z)t}
\end{pmatrix} :
z \in \mathbb{C} \textrm{ fixed} 
$$
here the $\cdot$ can be zero or any element of $\mathbb{R}$. The first and last are infinite families of groups, since we can choose any fixed $a, b \in \RR$ or $z \in \CC$.  The second item is only one group. %We need prove that limits of these groups cannot increase in dimension. 
The first two families are definable where the O-minimal structure defined including real exponential functions, so by \cite{Dries} limits cannot increase in dimension. 
%
%For the first two flavors of group, notice that they have the property of being a \emph{definable group}, see Coste or van den Dries.  Van den Dries Prop 3.1 the dimension cannot increase in the limit for Hausdorff limits.  Roughly speaking, the Chabauty topology coincides with the Hausdorff topology.  So we intersect with large compact balls, and take limits as these balls get larger.  This shows the limits we have computed for the first two kinds of group are the only possible ones. 
%
The last family of group is not definable, since $e^{zt}$ is not definable in any O-minimal structure.   There are (up to taking transpose) two families of groups of this sort: \[W_{1,2} ^z := \begin{pmatrix}  e^ {zt} & 0 \\
0 & e ^{-2\mathfrak{R}(z)t}
\end{pmatrix}, \qquad \qquad \text{ and }\qquad \qquad \;W_{3,8/9} ^z:=\begin{pmatrix}  e^ {zt} & \ast \\
0 & e ^{-2\mathfrak{R}(z)t}
\end{pmatrix}.\]

%\iffalse 
%The last flavor of group has the property of being \emph{uniformly proper}. %{\color{blue}\sout{ That is, distances in conjugates of the group are roughly the same as distances in the larger ambient group. }}
%  A group $H \leq G$ is \emph{NOT uniformly proper} if there exists a sequence $g_n \in G$ and $h_n \in H$ such that $g_n h_n g_n ^{-1} \to Id_G$ but $ h_n  \not \to Id_H$ %\todo{ I think it should be $h_n \not\to Id_H$}, 
%  in the sense that either for all open neighborhoods $U \subset G$ there exists $N $ such that $\forall n\geq N$ we have $h_n \in (U \cap g_n H g_n ^{-1} ) _0$.  Alternatively, we can take a left invariant Riemannian metric on $G$ and look at $d_H (h_n , id)$, the induced distance.  %A group is not uniformly proper if distances in $H$ are much farther than distances in $G$. 
%  \fi

  %
 We will show both these families of groups have limits which stay constant in dimension.  
 Lemma \ref{W12} 
 %The next lemma, which mirrors the idea in the proof of \cite{CDW} Theorem 3.1,  
 treats the one parameter group $W_{1,2} ^z$ and Lemma \ref{W389} completes the proof for $W_{3,8/9} ^z$. 
Both lemmas are corollaries of Theorem \ref{1d limits}.

  The inspiration for Theorem \ref{1d limits} is \editA{\cite[section 3.2]{CDW}} which gives the following example of a limit $p_n H p_n^{-1} \to L$ :  
  \[  H= \begin{pmatrix} 
  1 & t & 0 & 0 \\
  0 & 1 & 0 & 0 \\
  0 & 0 & \cos t & \sin t \\
  0 & 0 & - \sin t & \cos t 
  \end{pmatrix}
  \qquad 
  p_n = 
  \begin{pmatrix} n^{-1} & 0 & 0 & 0\\
  0 & n & 0 & 0 \\
   0 & 0 & 1 & 0 \\
   0 & 0 & 0 & 1
  \end{pmatrix} 
  \qquad 
  L = \begin{pmatrix} 
  1 & s & 0 & 0 \\
  0 & 1 & 0 & 0 \\
  0 & 0 & \cos t & \sin t \\
  0 & 0 & - \sin t & \cos t 
  \end{pmatrix}. 
   \] 
  This example satisfies $\dim (L) = 2> 1=\dim(H)$. The group $H$ is a one-parameter subgroup of $L\simeq\RR \times \mathbb{S}^1$  that looks like a helix. Conjugating by $p_n$ coils the helix more tightly so that the orbits of $H$ accumulate. 
One obstruction for such a phenomenon is given by Theorem \ref{1d limits}, which we prove next. 

%\begin{proposition}\label{1d limits}
%	Let $X\in \frakgl_d(\CC)$ be a matrix with an eigenvalue which is not purely imaginary. Then the local conjugacy limits of the one-dimensional closed subgroup $H=\{\exp(tX) \;|\; t\in \RR \} \le \GL_d(\CC)$ are one-dimensional. \todo{N: removed the moreover part, as we don't really use it}	
%%Moreover, the same conclusion holds for one dimensional closed subgroups $H=\{\exp(tX) \;|\; t\in \RR \} \le \SL_d(\CC)$ where $X\in \fraksl_d(\CC)$ is a matrix with an eigenvalue which is not purely imaginary. 
%\end{proposition}

\begin{proof}[Proof of Theorem \ref{1d limits}]
	Let us denote by $B$ the Borel subgroup of $\GL_d(\CC)$ of upper triangular matrices, and let $\frakb$ be its corresponding Lie subalgebra. Since we work over $\CC$ we can replace $X$ by its conjugate Jordan form, which is an upper triangular matrix. We therefore assume without loss of generality that $X\in \frakb$.
	Let us denote by $\alpha$ the non purely imaginary eigenvalue  of $X$, and by $v\in \CC^d$ its eigenvector.
	
	The map $t\mapsto \|\exp(tX)v\| = \|e^{t\alpha} v\|=e^{t \Re(\alpha)}\|v\|$ is a homeomorphism since $\Re(\alpha)\ne 0$. It follows that the map $\RR \to \GL_d(\CC)$  given by $t\mapsto \exp(tX)$ is proper, and $H$ is closed. 

	Now let $g_n \in \GL_d(\CC)$ be a sequence such that $g_n H g_n^{-1} \to L$ in $\Sub(\GL_d(\CC))$, and suppose for contradiction that $\dim L > \dim H=1$. 
	Denote by $H_n = g_n H g_n^{-1}$ and $X_n= g_n X g_n ^{-1}$. So $g_n H g_n^{-1} = \{ \exp (tX_n) | t\in \RR\}$.
	By the Iwasawa decomposition $\GL_d (\CC) = U_d  \cdot B $ where $U_d  $ is the unitary group.
	Write $g_n = u_n b_n$ for $u_n\in U_d  $ and $b_n \in B$.  Since the unitary group is compact, we may assume, up to passing to a subsequence, that $u_n \to u$.  Thus $b_n H b_n^{-1} \to u^{-1} L u$.  Therefore, we may assume without loss of generality that $g_n \in B$, and hence also $X_n = g_n X g_n^{-1} \in \frakb$.

	It follows from $\dim L > \dim H$ that for every small enough identity neighborhood $I\in V \subset \GL_d(\CC)$, the number of components of $V \cap H_n $ goes to infinity.
	In particular, $V \cap H_n $ has more than one component for all large enough $n$.  So $H_n$ leaves $V$ and returns to it, and we have 
	\begin{equation}
	\text{ there exist $0<s_n<t_n$ such that $\exp(s_n X_n)\notin V$ but $\exp(t_n X_n) \in V$.} \tag{$\spadesuit$}\label{property}
	\end{equation}
	
	Fix some norm $\|\cdot \|$ on $\frakgl_d(\CC)$ and let $B_\varepsilon $ be the ball of radius $\varepsilon $ in $\frakgl_d(\CC)$ around $0$. Denote by $V_\varepsilon  = \exp(B_\varepsilon )$. 	
	The exponential map $\exp:\frakgl_d(\CC)\to\GL_d(\CC)$ is a local homeomorphism at $0$. Let $\varepsilon _0$ be small enough so that $\exp:B_{\varepsilon _0} \to V_{\varepsilon _0}$ is a homeomorphism and $V_{\varepsilon _0}$ is open.
	These neighborhoods have the following \textbf{trapping property:} 
	for every $0<\delta<\varepsilon _0$, and $0<t$ and $Y\in \frakgl_d(\CC)$, if $\exp(sY)\in V_{\varepsilon _0}$ for all $s\in [0,t]$ and $\exp(t Y) \in V_\delta$ then $\exp(sY)\in V_\delta$ for all $s\in [0,t]$. 
	The following upgraded form of \eqref{property} follows: for every $\delta>0$ and large enough $n$,
	\begin{equation}
	\text{ 	there exist $0<s_n<t_n$ such that  $\exp(s_n X_n)\notin V_{\varepsilon _0}$ but $\exp(t_n X_n) \in V_\delta$. } \tag{$\clubsuit$}\label{property2}
	\end{equation}

	Let $\varepsilon  = \frac{\varepsilon _0}{2}$. For every $\delta\in (0,\varepsilon )$, let $t_n>0$ be the first return of the curve $t \mapsto \exp(t X_n)$ to $\overline{V_\delta}$ after leaving $V_{\varepsilon _0}$. 
	\begin{claim} 
		$\exp(\frac{t_n}{2} X_n) \notin V_{\varepsilon }$. 
	\end{claim}
	Assume for contradiction that $\exp(\frac{t_n}{2} X_n) \in V_{\varepsilon }$. 
	We first show that $\exp(\frac{t_n}{2} X_n) \in \overline{V_\delta}$.
	Indeed, let $Y\in B_\epsilon$ be such that $\exp(Y)=\exp(\frac{t_n}{2} X_n)$. For all $s\in[0,2]$, $\exp(sY) \in V_{\epsilon_0}$ (since $\epsilon = \epsilon _0 /2$), and  \[\exp(2Y) = \exp(Y)^2=\exp(\frac{t_n}{2} X_n)^2=\exp(t_n X_n)\in  \overline{V_\delta}.\] It follows by the trapping property that $\exp(Y)=\exp(\frac{t_n}{2} X_n) \in \overline{V_\delta}$. 
	
	Recall that $t_n$ is defined as the first return of $t\mapsto \exp(t X_n)$ to $\overline{V_{\delta}}$ after leaving $V_{\varepsilon _0}$. Therefore $\exp(s X_n) \in V_{\varepsilon _0}$ for all $s\in [0,\frac{t_0}{2}]$.
	By the trapping property, it follows that $\exp(s X_n) \in \overline{V_{\delta}}$ for all $s\in [0,\frac{t_0}{2}]$. Hence, $\exp(s X_n) = \exp(\frac{s}{2} X_n)^2 \in \overline{V_{\delta}}^2 \subseteq V_{\varepsilon _0}$ for all $s\in [0,t_0]$. This contradicts the assumption that the curve leaves $V_{\varepsilon _0}$, and proves the claim.
	
	Thus we may further upgrade \eqref{property2}. For every $\delta\in(0,\varepsilon )$, and large enough $n$,
	\begin{equation}
	\text{ 	there exist $0<t_n$ such that  $\exp\left(\frac{t_n}{2} X_n\right)\notin V_{\varepsilon }$ but $\exp(t_n X_n) \in V_\delta$. } \tag{$\diamondsuit$}\label{property3}
	\end{equation} 
	
	By choosing an appropriate subsequence, assume $h_n = \exp\left(\frac{t_n}{2} X_n\right)$ satisfies \eqref{property3} for $\delta = \frac{1}{n}$. In particular, $h_n \nrightarrow 1$ but $h_n^2 \to 1$.
	Let $\alpha_1=\alpha, \alpha_2, \ldots, \alpha_d$  be the eigenvalues of $X$ and thus also of $X_n = g_n X g_n^{-1}$. By definition of $h_n$ it follows that $e^{t_n\alpha }$ is an eigenvalue of $h_n^2= \exp(t_n X_n)$. Since $h_n^2\to 1$, we have $e^{t_n\alpha} \to 1$. From the assumption $\Re(\alpha)\ne 0$  it follows that $t_n \to 0$.
	It follows that the eigenvalues of $h_n$, namely $e^{\frac{t_n}{2}\alpha_1}, \ldots, e^{\frac{t_n}{2}\alpha_d}$ tend to $1$ as well.

	%{ \blue Let $\{\alpha_1, \ldots, \alpha_d\}$  be the eigenvalues of $X$ and thus also of $X_n = g_n X g_n^{-1}$. By definition of $h_n$ it follows that $e^{t_n\alpha_i }$ is an eigenvalue of $h_n^2= \exp(t_n X_n)$. Recall $\alpha$ is an eigenvalue that is not purely imaginary, and $\alpha= \alpha_i$ for some $i$\todo{N: why should we introduce an unnecessary index here when we can set it to be the first of the list}. Since $h_n^2\to 1$, we have $e^{t_n\alpha} \to 1$. Since $\Re(\alpha)\ne 0$  it follows that $t_n \to 0$.
	%It follows that all the eigenvalues of $h_n$, namely $e^{\frac{t_n}{2}\alpha_1}, \ldots, e^{\frac{t_n}{2}\alpha_d}$ tend to $1$ as well. } 
	
	We get a sequence of matrices $h_n\in B$ satisfying
	\begin{equation}
		\text{ $h_n^2 \to 1$, $h_n \nrightarrow 1$, and all eigenvalues of $h_n$ tend to $1$.} \tag{$\heartsuit$}\label{property4}
	\end{equation} 

Let us suppress indices and write $h=h_n$. The matrix $h$ is upper triangular so we may write 
\[ h=  \begin{pmatrix}  x _{11} & x_{12} & \cdots & x_ {1d}  \\
0 & x_{22} & \cdots & x_{2d} \\
 \vdots &  & \ddots & \vdots  \\
 0 & 0 & \cdots & x_{dd} 
\end{pmatrix}
\qquad 
h^2=  \begin{pmatrix}  x _{11} ^2 & x_{12}(x_{11} + x_{22}) & \cdots & \sum_{k=1}^d  x_{1k}x_{kd}   \\
0 & x_{22} ^2  & \cdots & \sum_{k=1}^d  x_{2k}x_{kd}  \\
 \vdots &  & \ddots & \vdots  \\
 0 & 0 & \cdots & x_{dd} ^2 
\end{pmatrix}. 
 \] 
By assumption, as $n\to \infty$, the eigenvalues of $h$ tend to $1$, i.e $x_{ii}\to 1$ for all $1 \leq i \leq d$.  We also have $h^2\to 1$ as $n\to\infty$. In particular, looking at the super-diagonal entries of $h^2$ we have $x_{i, i+1} (x_{ii} + x_{i+1, i+1}) \to 0$. But since $x_{ii}+x_{i+1,i+1} \to 2$ we can deduce that $x_{i,i+1} \to 0$ for all $1\le i \le d-1$. Proceeding by induction to next diagonal, we get $x_{i,j} \to 0$ for all $1\le i<j\le d$. We conclude that $h \to 1$ as $n\to\infty $. However this contradicts \eqref{property4}, and completes the proof.
%Since $\SL(d, \CC) \leq \SL(d, \CC)$ consists of matrices with determinant 1, as long as $X$ satisfies the hypotheses, then exactly the same proof holds for $H \leq \SL(d, \CC)$. 
\end{proof}

%Since $d=3$ is odd, Proposition \ref{1d limits} holds for all 1-parameter subgroups of $\SL_3(\CC)$. \todo{why's that? not all 1-par subgroups are even closed... I would remove this sentence.}

  \begin{lemma} \label{W12} Limits of $W_{1,2} ^z$ are 1 dimensional. 
  \end{lemma} 
  \begin{proof} 
  The group	$W_{1,2}^z$ is given by $\{\exp(tX)\;|\;t\in \RR\}$ for 
  \[X=\begin{pmatrix}  
  a & b  & 0 \\ 
  -b & a& 0 \\ 
  0 & 0 &-2a  
  \end{pmatrix}\in \fraksl_3(\RR)\]
  with $a\ne 0, b \in \RR$.
  This matrix has a non-zero real eigenvalue. As a subgroup of $\GL_3(\CC)$ this satisfies the assumptions of Theorem \ref{1d limits}, and therefore every local limit of this subgroup under conjugation in $\GL_3(\CC)$ is one-dimensional. In particular, the same conclusion also holds for conjugates of the closed subgroup $ W_{1,2} ^z $ in $ \SL_3(\RR)$, since $\SL_3(\RR)$ is a closed subgroup of $\GL_3(\CC)$.
  \end{proof} 
  
%
%\iffalse
%\begin{lemma} If $H$ is uniformly proper, and $g_n \in G$ is a sequence such that $H^{g_n} \to L$, then $ \dim H = \dim L$.  
%\end{lemma}
%\begin{proof} Put $H_n = H^{g_n}$.   Suppose $\dim L > \dim H$.  Then there exists a neighborhood of the identity $U$ such that the number of connected components of $H_n \cap U$ goes to infinity.  But this contradicts that the action is uniformly proper. \end{proof}  
%
%
%  Notice that if a {\color{blue} 1-dimensional sub}group has positive translation length {\color{blue}with respect to the action on the symmetric space associated to $\SL_3(\RR)$} then it is uniformly proper. 
%  We are now ready to prove Proposition \ref{samedim}.  The first two flavors of non-algebraic groups are definable, so their limits have the same dimension.  There are two isomorphism classes groups of the last flavor: $W_{1,2} ^z$ and $W_{3, 8/9} ^z$. %Any group with non-unit weights cannot be conjugated close to the identity, because weights are a conjugacy invariant by Proposition \ref{char}.  %Thus limits of $W_{1,2} ^z$ cannot increase in dimension, and we have computed them all.  
% Limits of $W_{1,2} ^z$ cannot increase in dimension since it has a positive translation length and acts uniformly properly. Lemma \ref{W389} proves the same for $W_{3, 8/9} ^z$.  This completes the proof of Proposition \ref{samedim}.  \todo{I think stating it like this is OK, but let me know.} 
% \fi 
 
% \todo{I think it makes more sense to prove both parts together or split both parts to Lemmas which will appear in the corresponding section, like Lemma \ref{W389}}

	\begin{lemma}\label{W389} Limits of  $W_{3,8/9} ^z$ are 3 dimensional. 
	\end{lemma} 
	%{\purple FIX PROOF ACCORDING TO REFEREE COMMENTS }
	
	\begin{proof} 
		Set $H=W^z_{3,8}$ for some $z\notin i\RR$, and assume that $p_n H p_n^{-1} \to L$ we wish to show that $\dim L=3$. We first claim we can reduce to the case that the $p_n$ are upper triangular. Use the Iwasawa decomposition to write $G=KB$ where $K=\SO(3)$ and $B$ is the subgroup of upper triangular matrices in $\SL_3(\RR)$. Hence we can write $p_n = k_n b_n$ for $k_n\in K$ and $b\in B$. By compactness of $K$ we may assume, up to passing to a subsequence, that $k_n \to k$. Thus $b_n H b_n^{-1} \to k^{-1} L k$. Since $\dim k^{-1} L k = \dim L$, the claim follows.
		
		Now, both $H$ and $B$ are contained in the parabolic subgroup 
		\[Q = W_{6,1}= \begin{pmatrix} 
		* & *  & * \\
		* & *  & * \\
		0 & 0 & *
		\end{pmatrix},\]
		and hence it suffices to look at limits in $\Sub(Q)$. 
		Let $p: Q \to \GL_2 (\RR)$ be the homomorphism sending a matrix in $Q$ to its upper left $2\times2$ block.  The homomorphism $p$ induces a pullback map $p^*: \Sub(\GL_2(\RR)) \to \Sub(Q)$, by $p^*(\tilde{A})=p^{-1}(\tilde{A})$. The map $p^*$ is a homeomorphism between $\Sub(\GL_2(\RR))$) and $\{A  \in \Sub(Q) | A  \ge \ker p\}$, as it has an inverse $p_*: \{ A \in \Sub(Q) | A  \ge \ker p\} \to \Sub(\GL_2(\RR))$ defined by $p_*(A)=p(A)$. The group $H$ contains $\ker p= \left(\begin{smallmatrix} 1 & 0 & *\\ 0 & 1 & *\\ 0 & 0 & 1 \end{smallmatrix}\right)$.
		The image of $H$  under this homeomorphism is the 1-dimensional subgroup $\tilde{H}=(e^{zt}) \in \Sub(GL_2(\RR))$. %By a similar argument to the group $W_{1,2}^z$ we see that a limit $\bar{L}$ of conjugates of $H$ must be 1-dimensional.  
		%Again, $\tilde H$ acts with positive translation length on the associated symmetric space \todo{not good... we should replace this argument with an instance of Proposition \ref{1d limits}. Can you do that? you can change Proposition \ref{1d limits} to $\mathfrak{gl}_d(\CC)$ for that purpose} $\HH^2 \times \RR$ of $\GL_2(\RR)_0\simeq \SL_2(\RR) \times \RR$. As in the proof of the previous lemma, limits of conjugates of $\tilde H$ in $\Sub(\GL_2(\RR)$ are 1 dimensional.
		 By Theorem \ref{1d limits} limits of $\tilde H$ in $\Sub(\GL_2(\RR)$ are 1 dimensional since $z \not \in i \RR$ by assumption.  
		Let $L$ be a limit of conjugates of $H$, then $L\le p^*(\tilde{L})$ for some limit  $\tilde{L}$ of conjugates of $\tilde{H}$. Since $\ker p$ is 2-dimensional, and $\bar L$ is 1-dimension, then $p^*(\tilde{L})$ is 3-dimensional. Hence $L$ is 3-dimensional.
	\end{proof} 

	We will also extensively use the following propositions from \cite{CDW, CLV} to identify which subgroups cannot limit to other subgroups.

 Denote the normalizer of a subgroup $H \leq G$ by $N_G(H)$.  Denote the connected component of the identity by $H_0$.  The next theorem says that the dimension of the normalizer increases under taking a limit. 

\begin{proposition}[\editA{\cite[Proposition 3.2]{CDW}}]\label{normalizer} Let $G$ be an algebraic Lie group (defined over $\CC$ or $\RR$), let $H$ be an algebraic
subgroup and let $L$ be any limit of $H$. Then $\dim N_G(H_0) \leq \dim N_G(L_0)$ with equality if and only if
$L$ and $H$ are conjugate.
\end{proposition} 

The same statement works for non-algebraic groups as well as long as the dimension does not increase when taking a limit.  In this case, taking a local limit of groups by conjugation is equivalent to taking a limit of their Lie algebras by $\Ad(G)$ action, and the same proof idea works.

%\todo{we should explain why this works for non-algebraic groups as well, as long as the dimension of the limit is the same as the subgroup. So in fact, we should explain the connection between limits with the same dimension and Lie subalgebra limits, and claim that their proof works because it is based on the Lie subalgebra argument. We should not use $G_0$ etc. but simply say that $H$ locally converges to $L$ by conjugation}

Any element $A \in \mathfrak{gl}(n)$ has a well defined characteristic polynomial, denoted $char(A)$.
Given a Lie subalgebra $\mathfrak{h} \subseteq \mathfrak{gl}(n)$, %\todo{we should use $\subseteq$ also in the next theorem, both for Char and for gl(n)}
we denote by $Char(\mathfrak{h})$ the closure of the subset $\{ char(A) : A \in \mathfrak{h} \} \subset \RR[x]$ . Thus $Char(\mathfrak{h})$ is closed and invariant under
conjugation of $\mathfrak{h}$. The next proposition implies that limits 
have smaller sets of characteristic polynomials. 

\begin{proposition}[\editA{\cite[Proposition 3.4]{CDW}}]\label{char}
Suppose $H$ is a closed algebraic subgroup of $\GL_n(\RR)$, and $L$ is a conjugacy limit
of $H$. Then $Char(\mathfrak{l}) \subseteq Char(\mathfrak{h})$, where $\mathfrak{h}, \mathfrak{l}\subset  \mathfrak{gl}(n)$ %\todo{change to $gl_n(\RR)$} 
denote the Lie algebras of $H$ and $L$ respectively.
\end{proposition}

The next proposition implies that limits of abelian groups are abelian.  A group $H$ satisfies a universal relation
if there is a finitely generated free group $F$ and a word $w \in F$ such that for all homomorphisms
$\theta : F \to H$ we have $\theta (w) =1$. 

\begin{proposition}[\editA{\cite[Proposition 2.2]{CLV}} idea due to Daryl Cooper] \label{univ} If $H \leq G$ satisfies a universal relation, $w$, then so does every $G$-conjugacy limit
$L$ of $H$.
\end{proposition}

%Now it remains to deal with $W_{3,8} ^z$ since $W_{3,9} ^z$ is the transpose and follows the same argument.   Recall we are only concerned with conjugating by sequences in the Borel, and which are not in the normalizer.  So it suffices to conjugate only by sequences with non-identity entries in the upper left block which are upper triangular.  We can show that there is a homomorphism to project to this block, and that doing so is equivariant with respect to conjugacy.  Using the previous results about uniformly proper for $e^{zt}$ in this block shows the dimension of the limit in this block must stay one dimensional, and thus limits of $W_{3,8} ^z$ are three dimensional.  \textcolor{red}{make this more precise}. 

To organize the local limit charts of Theorem \ref{main}, we note that by Proposition \ref{samedim} we can treat each dimension separately. In view of  Proposition \ref{normalizer}, normalizers of (non-conjugate) limits must increase in dimension, it is therefore convenient to arrange the columns (or rows) of the chart by the dimension of the normalizer. Arrows thus can only go to the right (or down if arranged by rows).  %Next, notice that every group in Winternitz's list is definable.  In fact, they are algebraic except for three families, which are products with the same 1-dimensional groups.  So we may apply Theorem \ref{algebraic} and \cite{Dries} Proposition 9.2 to say that these groups must limit only to groups of the same dimension.   
To complete the proof in each section, we need to provide a conjugating sequence of matrices for each arrow in the chart, and prove nonexistence of any arrows from left to right (up to down), which we do using the remainder of the theorems and propositions from this section. 

\section{The classification of local limits in $\SL_3(\RR)$}\label{classification}

\subsection*{Dimension 1}%%%%%%%%%%%%%%%%%%

 We begin with a table of the one-parameter subgroups of $\SL_3(\RR)$ and their properties. After the table, the terminology \emph{singular} and \emph{non-singular} and other conventions used in the table are explained.

%There are five 1 dimensional subgroups.   
$$\begin{array}{c|c|c|c} 
\textrm{Name} & \textrm{Group} & \textrm{Normalizer} %& \dim (\textrm{Normalizer}) 
&\textrm{Properties} \\  
\hline 
W_{1,1}^{(a,b)} & \begin{array}{c}  \begin{pmatrix} e^{at} & 0 & 0 \\ 0 & e^{bt} & 0 \\ 0 & 0 & e^{-(a+b)t} \end{pmatrix}\\ (a,b)\in\mathbb{R}\setminus\{0\} \end{array}  & \begin{array}{c} 
%\begin{pmatrix} * & 0 & 0 \\
%0 & * & 0 \\
%0 & 0 & *
%\end{pmatrix} \\
W_{2,2} \text{ if  nonsingular}  \\
% \begin{pmatrix}  GL_2 & 0 \\
%0 & \frac{1}{det}
%\end{pmatrix} \\
W_{4,1}\text{ if singular}
  \end{array} 
  %& 
 % \begin{array} {c} 
 % 2 \\  4 
%  \end{array}
 &
\begin{array}{c} 
\textrm{definable} \\ \textrm{algebraic}
\end{array} \\
\hline 
W_{1,2} ^z & \begin{array}{c}  \begin{pmatrix} e^{zt} & 0 \\ 0 & e^{-2\Re(z)t} \end{pmatrix} \\ z\in \mathbb{C} \end{array}  & 
%\begin{pmatrix}  \mathbb{C} ^* & 0 \\
%0 & \frac{1}{det} 
%\end{pmatrix} 
W_{2,1}
%& 2 
& \begin{array}{c}
\text{limits 1 dimensional} \\ \text{algebraic if }z\in i\RR \end{array}\\
\hline 
W_{1,3} & 
\begin{pmatrix} e^{t} & te^{t} & 0 \\ 0 & e^{t} & 0 \\ 0 & 0 & e^{-2t} \end{pmatrix} 
& 
W_{2,3}
%\begin{pmatrix} a & * & 0 \\
%0 & a & 0 \\
% 0 & 0 & a ^{-2} 
%\end{pmatrix}  
%& 2 
& \textrm{definable} \\
\hline 
%\end{array} $$
%$$ \begin{array}{c|c|c|c} 
%\hline
W_{1,4} & \begin{pmatrix} 1 & t & 0 \\ 0 & 1 & 0 \\ 0 & 0 & 1 \end{pmatrix} & 
%\begin{pmatrix}  * & * & * \\
% 0 & * & 0 \\ 
% 0 & * & *
%\end{pmatrix} 
\text{conjugate to }W_{5,3}
 %&  5 
 & \textrm{algebraic} \\
\hline 
W_{1,5} & \begin{pmatrix} 1 & t &\frac{ t^2}{2} \\ 0 & 1 & t \\ 0 & 0 & 1 \end{pmatrix}
  & 
%\begin{pmatrix} 
%a & ax & *  \\
%0 & 1 & x\\
%0 & 0 & \frac{1}{a}
%\end{pmatrix} 
W_{3,7}&
 %& 3 &
  \textrm{algebraic } \\
\end{array} $$

 The group $W_{1,1}^{(a,b)}$ depends on a choice of two real numbers $a, b \in \RR$. % If none of the weights match, the group is \emph{nonsingular}.  If any of the weights match, it is \emph{singular}, i.e., $a=b$, $a = -a-b$ or $b = -a-b$.
 \editA{If none of the weights $a$, $b$ and $-a-b$ match, the group is \emph{nonsingular}.  If any of the weights match, i.e. if $a=b$, $a = -a-b$ or $b = -a-b$, it is called \emph{singular}.} 
	% Recall $a,b \in \RR$ are fixed.  Here we say $W_{1,1}^ {(a,b)}$ is \emph{singular} if some of the weights match: either $a=b, a= -(a+b)$ or $b = -(a+b)$. 
	We abuse notation and denote the nonsingular case by $W_{1,1}^ {a \neq b}$ and the singular case $W_{1,1}^ {a=b}$, and use this notation throughout the rest of the paper. Note that if $z\in \RR$ then $W_{1,2}^z = W_{1,1}^{(a,a)}$. Therefore we abuse notation in writing $W_{1,2} ^z$ to assume also $z\notin \RR$. 

The possible local limits of each group are represented in the following transitive chart. 
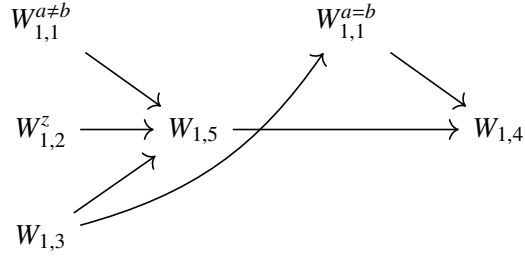
\begin{center}
\begin{tikzcd} 
W_{1,1}^{a \neq b} \arrow[rd] & & W_{1,1}^{a=b} \arrow[rd]  \\
 W_{1,2}^z  \arrow[r ]   & W_{1,5} \arrow[rr]&  & W_{1,4}\\
W_{1,3} \arrow[ru] \arrow[rruu,bend right= 20]\\
\end{tikzcd} 
\end{center}
Another way to see the local limits of one parameter groups is to take a sequence of elements in the groups that converges to the limit.  However, we wanted all the sections of the paper to be consistent. So,
we give a sequence of conjugating matrices,  $p_n$ in the sense of Definition \ref{conv}, for each arrow that appears in the chart: 
\[ \begin{array}{c} W_{1,5} \to W_{1,4}\\  
\begin{pmatrix}
\frac{1}{n} & 0 & 0\\
0 & \frac{1}{n} & 0 \\
0& 0 & n^2 
\end{pmatrix}
\end{array} 
\qquad
%W_{1,3} \to W_{1,4}: 
%\begin{pmatrix}
%n& 0 & 0\\
%0 & 1 & 0 \\
%0& 0 & \frac{1}{n}
%\end{pmatrix}
%\qquad 
\begin{array} {c} 
W_{1,3} \to W_{1,5}: \\
\begin{pmatrix}
n& 0 & \frac{n}{9}\\
0 & 1 & \frac{-1}{3} \\
0& 0 & \frac{1}{n}
\end{pmatrix}
\end{array} 
\qquad 
\begin{array}{c} 
W_{1,1}^{a\neq b}  \to W_{1,5} : \\
\begin{pmatrix}
1& n & \frac{(a - b)^2 n^2}{2 a^2 + 5 a b + 2 b^2}\\
0 & 1 & \frac{(a - b) y}{a + 2 b} \\
0& 0 & 1
\end{pmatrix}
\end{array} 
\]

\[
%\begin{array}{c} 
%W_{1,1}^{a=b}  \to W_{1,4}: \\\
%\begin{pmatrix}
%1& n & 0\\
%0 & 1 & 0 \\
%0& 0 & 1
%\end{pmatrix}
%\end{array} 
\begin{array}{c} 
W_{1,1}^{a=b}  \to W_{1,4}: \\\
\begin{pmatrix}
1& 0 & n\\
0 & 0 & 1 \\
0& -1 & 0
\end{pmatrix}
\end{array}  
\qquad
%W_{1,2} \to W_{1,4}: 
%\begin{pmatrix}
%1& n & 0\\
%0 & 1 & 0 \\
%0& 0 & 1
%\end{pmatrix}
%\qquad 
%
%
\begin{array}{c} 
W_{1,2}^z  \to W_{1,5}:  \\
\begin{pmatrix}  n & \frac{3an}{b} & \frac{(9a^2 + b^2)n }{b^2} \\
0 & 1 & 0\\
0 & 0 & \frac{1}{n}
\end{pmatrix} 
\end{array} 
\qquad 
\begin{array}{c}
 W_{1,3} \to W_{1,1} ^{a=b}: \\
\begin{pmatrix}  \frac{1}{n} & 0 & 0 \\
0 & n &0 \\
 0 & 0 & 1
\end{pmatrix}
\end{array} \]

%\todo{checked for $z \in i \RR$ and same conjugation works} 
Recall we have assumed $z \not \in \RR$. Next we explain nonexistence of the missing arrows. The first subscript of the group in the normalizer column is the dimension of the normalizer.  Proposition \ref{normalizer} explains the missing arrows except for $ W_{1,5} \not \to W_{1,1}^{a=b}$, $W_{1,1}^{a\neq b} \not \to W_{1,1}^{a=b}$, and $W_{1,2} ^z \not \to W_{1,1}^{a=b}$ which follow from Proposition \ref{char}.

\subsection*{Dimension 2}%%%%%%%%%%%%%%%%%%%%

$$\begin{array}{c|c|c|c} \textrm{Name} & \textrm{Group} & \textrm{Normalizer}  %& \dim (\textrm{Normalizer}) 
&\textrm{Properties} \\  
\hline 
W_{2,1} & \begin{pmatrix} \mathbb{C}^* & 0 \\ 0 & \editA{\det^{-1} \nolimits} \end{pmatrix} & W_{2,1} & %2 &
 \cong \textcolor{purple}{\mathbb{C}^*}\\
\hline 
W_{2,2} & \begin{pmatrix} * & 0 & 0 \\ 0 & * & 0 \\ 0 & 0 & * \end{pmatrix} &  W_{2,2} & %2& 
  \cong (\RR ^2, +)  \\
\hline 
%\end{array} 
%$$ 
%$$\begin{array}{c|c|c|c}
W_{2,3} &  \begin{pmatrix} e^{t} & * & 0 \\ 0 & e^{t} & 0 \\ 0 & 0 & e^{-2t} \end{pmatrix}
& 
%\begin{pmatrix} * & * & 0 \\
% 0 & * & 0 \\
%  0 & 0 & *
%\end{pmatrix} 
\textrm{conjugate to } W_{3,1} 
%& 3 
&  \cong (\RR ^2, +) \\
\hline 
\end{array}
$$
$$
\begin{array}{c|c|c|c}\hline 
W_{2,4} &  \begin{pmatrix} 1 & * & * \\ 0 & 1 & 0 \\ 0 & 0 & 1 \end{pmatrix}
& 
W_{5,2} 
%\begin{pmatrix} * & * & * \\
%0 & * & * \\
%0 & * & *
%\end{pmatrix} 
& 
%6  &
  \cong (\RR ^2, +) \\
\hline 
W_{2,5} & 
 \begin{pmatrix} 1 & 0 & * \\ 0 & 1 & * \\ 0 & 0 & 1 \end{pmatrix}
& 
%\begin{pmatrix} * & * & * \\
%* & * & * \\
%0 & 0 & *
%\end{pmatrix} 
W_{5,1} & %6 &
 \cong (\RR ^2, +) \\
\hline  
W_{2,6} & \begin{pmatrix} 1 & t & * \\ 0 & 1 & t \\ 0 & 0 & 1 \end{pmatrix}
& 
%\begin{pmatrix} e^t & * & *\\
%0 & 1 & * \\
% 0 & 0 & e^{-t} 
%\end{pmatrix} 
W_{4,6}^{(1,0)}
%& 4 
& \cong (\RR ^2, +) \\
\hline 
W_{2,7}^{(a,b)} &  \begin{pmatrix} e^{at} & * & 0 \\ 0 & e^{bt} & 0 \\ 0 & 0 & e^{-(a+b)t} \end{pmatrix} &   
%\begin{pmatrix} * & * & 0 \\
%0 & * & 0 \\
% 0 & 0 & * 
%\end{pmatrix}
\textrm{conjugate to } W_{3,1} &  %3  &
 \begin{array} {c}  \cong  \Aff(\mathbb{R}) \\ \textrm{algebraic} \\ a=b \\ \\ \textrm{definable} \\a \neq b \end{array} \\ 
 \hline 
W_{2,8} &  \begin{pmatrix} e^t & * & te^t \\ 0 & e^{-2t} & 0 \\ 0 & 0 & e^t \end{pmatrix}
& 
%\begin{pmatrix}  e^t & * & *\\
%0& e^t & 0 \\
%0 & 0 & e ^{-2t}
%\end{pmatrix}  
W_{3,2} & %3 &
  \begin{array} {c}  \cong  \Aff(\mathbb{R}) \\ \textrm{definable} \end{array} \\ 
\hline 
W_{2,9} &  \begin{pmatrix} e^{-2t} & 0 & * \\ 0 & e^t & te^t \\ 0 & 0 & e^t \end{pmatrix}
& 
%\begin{pmatrix} e^ {-2t} & 0 & *\\
%0 & e^t & * \\
% 0 & 0 & e^t 
% \end{pmatrix}
W_{3,3} & %3 &
 \begin{array} {c}  \cong  \Aff(\mathbb{R}) \\ \textrm{definable} \end{array} \\ 
\hline 
W_{2,10} &  \begin{pmatrix} e^t & e^t s & e^t \frac{s^2}{2} \\ 0 & 1 & s \\ 0 & 0 & e^{-t} \end{pmatrix} &  W_{2,10} & %2 &
 \begin{array} {c}  \cong  \Aff(\mathbb{R}) \\ \textrm{algebraic} \end{array} \\ 
\end{array} $$ 

Note that $W _{2,3} = W_{2,7} ^{a=b}$. 
The full chart of local limits in dimension 2 is

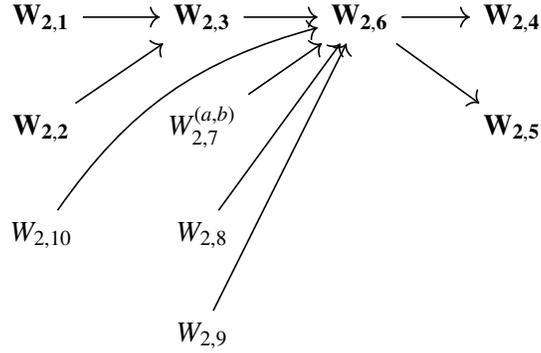
\begin{center}
\begin{tikzcd}  
\mathbf{W_{2,1}} \arrow[r] & \mathbf{W_{2,3}} \arrow[r]  & \mathbf{W_{2,6}} \arrow[r] \arrow[rd]& \mathbf{W_{2,4}} \\ 
\mathbf{W_{2,2}} \arrow[ru] & W_{2,7}^{(a,b)}  \arrow[ru] && \mathbf{W_{2,5}}  \\
W_{2,10} \arrow[rruu, bend left= 20]& W_{2,8} \arrow[ruu] && \\
& W_{2,9} \arrow[ruuu]&&
\end{tikzcd} 
\end{center} 
We have put the abelian groups in \textbf{bold} to distinguish them. 
%Haettel \cite{Haettel} (Proposition 5.11) shows the 2-skeleton of the cell complex constructed by gluing together the orbits of limits of $W_{2,2}$ is homeomorphic to a wedge of spheres and $\RR P^2$.  In general, determining the homeomorphism type of each part is difficult. 
%
Excluding limits of $W_{2,1}$ the computations for abelian groups appear in \cite{Leitnersl3, Haettel}.    %Since $W_{2,1}$ is also Cartan, (as is the diagonal), then it cannot limit to the diagonal.  
We first give the computations for the remainder of the arrows which do appear in the chart. To finish the limits of the abelian groups, we see $W_{2,1} \to W_{2,3}$ by $diag\langle n, 1, \frac{1}{n} \rangle$. 
%$$\begin{pmatrix}
%n& 0& 0\\
%0 & 1 & 0 \\
%0& 0 & \frac{1}{n}
%\end{pmatrix}.
%$$
%
%See Arielle's $SL(3,\mathbb{R})$ paper for the graph of limits for the abelian groups excluding the rotational group 
%
%\begin{tikzcd} 
%& & & W_{2,4} \\
%W_{2,2} \arrow[r] & W_{2,3} \arrow[r] & W_{2,6} \arrow[ru] \arrow[rd]\\
%& & & W_{2,5} 
%\end{tikzcd} 

%Notice $W_{2,7}^{(a,b)} \to W_{2,6}$ where we first conjugate $W_{2,7}$ be a permutation matrix that swaps the second and third columns, and then by 
Next we compute all of the limits of the nonabelian groups. The next two limits are done by first conjugating by a permutation matrix to move the free element to the upper right corner, and then applying the sequence shown. 
\[ \begin{array}{c} W_{2,7}^{(a,b)} \to W_{2,6}: \\
\begin{pmatrix} 
1 &  n & 0\\
0 & 1 & \frac{-a +b}{-a-2b} n \\
0 & 0 &1
\end{pmatrix}
\end{array} 
\qquad 
\begin{array}{c} 
W_{2,8} \to W_{2,6} :\\
 \begin{pmatrix} n &  0&  \frac{-2n }{9}  \\
0& 1&  \frac{-1}{3}\\
0& 0& \frac{1}{n}
\end{pmatrix}.  
\end{array} 
\qquad
%
%
%
%Do NOT see a way for $W^{a \neq b}_{2,7} \to W_{2,6}$.  But $W^{a \neq b}_{2,7} \to W_{2,4}$ by 
%$$\begin{pmatrix}
%%1& n& n\\
%0 & 1 & 0 \\
%0& 0 & 1
%\end{pmatrix}
%$$
%
%and $W^{a \neq b}_{2,7} \to W_{2,5}$ by 
%$$\begin{pmatrix}
%1& 0& 0\\
%0 & 1 & n \\
%0& 0 & 1
%\end{pmatrix}
%$$
%
%Do NOT see a way for $W_{2,8} \to W_{ 2,6}$.  But  $W_{2,8} \to W_{ 2,4}$ by 
%$$\begin{pmatrix}
%n& 0& 0\\
%0 & 1 & 0 \\
%0& 0 & 1
%\end{pmatrix}
%$$
%
%and $W_{2,8} \to W_{ 2,5}$ by 
%$$\begin{pmatrix}
%1& 0& 0\\
%0 & 1 & n \\
%0& 0 & 1
%\end{pmatrix}
%$$
%
%We conjugate $W_{2,8}$ by a permutation matrix which switches the second and third elements.  Then $W_{2,8}$ is a group with $*$ in the upper right corner.  In this conjugate of $W_{2,8}$ we can take a limit to $W_{2,6}$ by 
%$$ \begin{pmatrix} n &  0&  \frac{-2n }{9}  \\
%0& 1&  \frac{-1}{3}\\
%0& 0& \frac{1}{n}
%\end{pmatrix}.  
%$$
%
%Finally, we fill in the rest of the limits
\begin{array}{c}
 W_{2,9} \to W_{2,6} : \\
 \begin{pmatrix}\frac{9}{1 - n^3}& \frac{3 n^3}{1 - n^3}& 1\\
 0& n& 0\\
  0& 0& \frac{1 - n^3}{9 n}
\end{pmatrix}
\end{array} 
\qquad 
\begin{array}{c}  
W_{2,10} \to W_{2,6}:\\
\begin{pmatrix}
1& 0& n\\
0 & 1 & 0 \\
0& 0 & 1
\end{pmatrix}.
\end{array}
\] 

It remains to prove nonexistence of the missing arrows. By Proposition \ref{univ} limits of $W_{2,1}$ are abelian groups. 
%
%
%
%Now we compute limits for nonabelian groups. 
%Notice in the case $a=b$ then $W_{2,7} ^{(a,a)} = W_{2,3}$.  
%Since limits of abelian groups are abelian by Proposition \ref{univ}, 
%To finish we need only compute the possible limits for nonabelian groups.  
%By Proposition \ref{char} the nonabelian groups can only limit to unipotent groups, which finishes the argument. 
To finish the argument, we need to explain why there are no missing arrows.  It remains to check the non-abelian groups. All the missing arrows from non-abelian groups would originate from $W_{2,10}$, and Proposition \ref{char} does not allow an arrow to any of $W_{2,3}, W_{2,7}, W_{2,8}, W_{2,9}$, because the corresponding characteristic polynomials are not contained in $\textrm{Char}(W_{2,10})$.

%Notice $W_{2,10}$ cannot limit to any of the other nonalgebraic groups by Proposition \ref{char}. 

%\begin{tikzcd} 
%& W_{2,8} \arrow[rdd] & W_{2,7} \arrow[dd ] %\arrow[rd] \arrow[rddd] 
%&\\
%W_{2,1} \arrow[rd] & & & W_{2,4} \\
%W_{2,2} \arrow[r] & W_{2,3} \arrow[r] & W_{2,6} \arrow[ru] \arrow[rd]\\
%& & & W_{2,5} \\
%& W_{2,9} \arrow[ruu] & W _{2,10}  \arrow[uu]& 
%\end{tikzcd} 

\subsection*{Dimension 3}%%%%%%%%%%%%%%%%%%%%%%%%%%%

$$\begin{array}{c|c|c|c} \textrm{Name} & \textrm{Group} & \textrm{Normalizer} &% \dim (\textrm{Normalizer}) &
\textrm{Properties} \\  
\hline 
W_{3,1} &  \begin{pmatrix} * & 0 & * \\ 0 & * & 0 \\ 0 & 0 & * \end{pmatrix} &   W_{3,1} & %3& 
\begin{array}{c}  \cong \textrm{Aff}(\mathbb R)\times \mathbb R \\ \textrm{algebraic} \end{array}  \\ 
\hline 
W_{3,2} & \begin{pmatrix} e^t & * & * \\ 0 & e^t & 0 \\ 0 & 0 & e^{-2t} \end{pmatrix} & 
%\begin{pmatrix} * & * & * \\
 %0 & * & 0 \\
 % 0 & 0 & *
%\end{pmatrix} 
W_{4,2} & %4 &
  \begin{array}{c}  \cong \textrm{Aff}(\mathbb R)\times \mathbb R \\ \textrm{algebraic} \end{array} \\  
\hline 
W_{3,3} & \begin{pmatrix} e^{-2t} & 0 & * \\ 0 & e^{t} & * \\ 0 & 0 & e^{t} \end{pmatrix} &  
%\begin{pmatrix} * & 0& * \\
% 0 & * & *\\
%  0 & 0 & *
%\end{pmatrix}
W_{4,3}  & %4 &
  \begin{array}{c}  \cong \textrm{Aff}(\mathbb R)\times \mathbb R \\ \textrm{algebraic} \end{array}  \\  
\hline 
W_{3,4} & \begin{pmatrix} 1 & * & * \\ 0 & 1 & * \\ 0 & 0 & 1 \end{pmatrix} & 
%\begin{pmatrix} * & * & *\\
%0 & * & *\\
%0 & 0 & *
%\end{pmatrix} 
W_{5,3} & 
%5 & 
\begin{array}{c} \textrm{Heisenberg} \\ \textrm{algebraic} \end{array}  \\
\hline 
W_{3,5}^{(a,b)} &  \begin{pmatrix} e^{at} & * & * \\ 0 & e^{bt} & 0 \\ 0 & 0 & e^{-(a+b)t} \end{pmatrix}  & 
% \begin{pmatrix}  * & * & * \\
%0 & * & 0\\
%0 & 0 & * 
%\end{pmatrix}  
W_{4,2}& %4 & 
\begin{array} {c}  \cong \mathbb R ^2 \rtimes \mathbb R \\
\mathbb{R} \textrm{ acts as } W_{1,1} \\ \textrm{definable}  \end{array} \\
\hline
W_{3,6}^{(a,b)} &  \begin{pmatrix} e^{at} & 0 & * \\ 0 & e^{bt} & * \\ 0 & 0 & e^{-(a+b)t} \end{pmatrix}    &  
%\begin{pmatrix}  * & 0& * \\
%0 & * & *\\
%0 & 0 & * 
%\end{pmatrix} 
W_{4,3} & %4 &
\begin{array} {c}  \cong \mathbb R ^2 \rtimes \mathbb R \\
\mathbb{R} \textrm{ acts as } W_{1,1} \\ \textrm{ definable} \end{array} \\
\hline 
W_{3,7} & \begin{pmatrix} e^{t} & e^t s & * \\ 0 & 1 & s \\ 0 & 0 & e^{-t} \end{pmatrix} & W_{3,7} &%3 &
 \begin{array} {c}  \cong \mathbb R ^2 \rtimes \mathbb R \\
\mathbb{R} \textrm{ acts as } W_{1,1} \\ \textrm{algebraic}  \end{array} \\
\hline 
W_{3,8}^{(z)} & \begin{pmatrix} e^{zt} & \mathbb{C} \\ 0 & e^{-2\Re(z)t }\end{pmatrix}, z\in \mathbb{C} & W_{4,4} &%4 &
 \begin{array} {c}  \cong \mathbb R ^2 \rtimes \mathbb R \\
\mathbb{R} \textrm{ acts as } W_{1,2} \\\textrm{limits 3 dimensional} \end{array} \\
\hline 
W_{3,9}^{(z)} & \begin{pmatrix} e^{-2\Re(z)t} & \mathbb{C} \\ 0 & e^{zt} \end{pmatrix} , z \in \mathbb{C} &  W_{4,5} & %4& 
\begin{array} {c}  \cong \mathbb R ^2 \rtimes \mathbb R \\
\mathbb{R} \textrm{ acts as } W_{1,2} \\ \textrm{limits 3 dimensional} \end{array} \\ 
\hline 
\end{array}
$$
$$
\begin{array}{c|c|c|c}\hline 
W_{3,10} & \begin{pmatrix} e^{t} & te^t & * \\ 0 & e^{t} & * \\ 0 & 0 & e^{-2t} \end{pmatrix} &  
%\begin{pmatrix}  e^t& * & * \\
 %0 & e^t & * \\
 % 0 & 0 & e^{-2t}
%\end{pmatrix} 
W_{4,6}^{(1,1)}
& %4 &
 \begin{array} {c}  \cong \mathbb R ^2 \rtimes \mathbb R \\
\mathbb{R} \textrm{ acts as } W_{1,3} \\ \textrm{definable} \end{array} \\ 
\hline 
W_{3,11} & \begin{pmatrix} e^{-2t} & * & * \\ 0 & e^{t} & te^t \\ 0 & 0 & e^{t} \end{pmatrix} &
% \begin{pmatrix}  e^{-2t} & * & * \\
 %0 & e^t & * \\
 % 0 & 0 & e^t
%\end{pmatrix} 
W_{4,6}^{(-2,1)} & %4 &
 \begin{array} {c}  \cong \mathbb R ^2 \rtimes \mathbb R \\
\mathbb{R} \textrm{ acts as } W_{1,3} \\ \textrm{definable} \end{array} \\ 
\hline
W_{3,12} & \begin{pmatrix} * & * & 0 \\ * & * & 0 \\ 0 & 0 & 1 \end{pmatrix} & W_{4,1}  & %4 & 
\begin{array}{c}  \cong \SL_2 (\mathbb{R}) \\ \textrm{algebraic} \end{array}  \\
\hline 
W_{3,13} & \SO(2,1) & W_{3,13} & %3 &
 \begin{array}{c}   \cong \SL_2 (\mathbb{R}) \\ \textrm{algebraic} \end{array} \\
\hline 
W _{3, 14} & \SO(3) & W_{3,14} & %3 &
 \begin{array}{c} \cong \SO(3) \\ \textrm{algebraic} \end{array} 
\end{array} $$

We compactify notation to write a group and its transpose in the same line, for example: $W_{3, 8/9}$. We show the full chart of local limits is as follows:

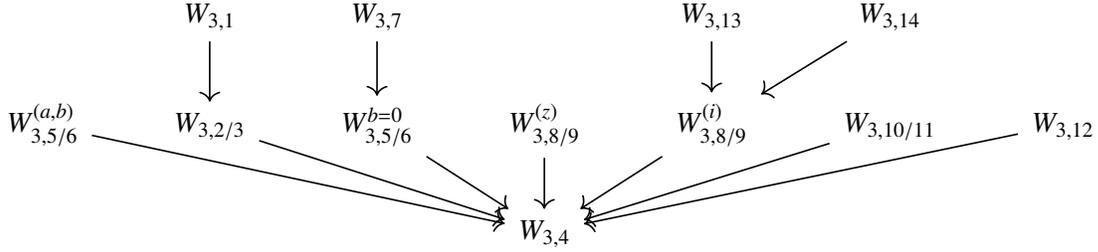
\begin{center} 
\begin{tikzcd} & W_{3,1}\arrow[d] & W_{3,7}\arrow[d] && W_{3,13}\arrow[d]& W_{3,14}\arrow[dl]	& \\ 
W_{3,5/6}^{(a,b)} \arrow[drrr] &W_{3, 2/3}\arrow[drr] & W_{3,5/6}^{b=0}\arrow[dr] & W_{3,8/9}^{(z)}\arrow[d]  & W_{3,8/9}^{(i)}\arrow[dl] & W_{3,10/11}\arrow[dll] & W_{3,12}\arrow[dlll] \\ 
&&& W_{3,4} &&&& 
\end{tikzcd} 
\end{center} 

%\begin{center}
%\begin{turn}{90}
%\begin{tikzcd} 
%		 				&	W_{3,5/6}^{(a,b)} \arrow[rddd]\\
%	W_{3,1}\arrow[r] 	&	W_{3,5/6}^{(a=b)} \arrow[rdd]\\
%	W_{3,7}\arrow[r]		&	W_{3,5/6}^{b=0}\arrow[rd]\\
%		 				&	W_{3,8/9}^{(z)}\arrow[r]	&	W_{3,4}\\
%	W_{3,13}\arrow[r]		&	W_{3,8/9}^{(i)}\arrow[ru]\\
%	W_{3,14}\arrow[ru]		&	W_{3,10/11}\arrow[ruu]\\
%						&	W_{3,12}\arrow[ruuu]
%\end{tikzcd} 
%\end{turn}
%\end{center}

Recall \cite{CDW} compute limits of $W_{3,8} ^{i} = \SO(2,1) = W_{3,13} $ and \cite{GJT} calculate limits of $\SO(3)=W_{3,14}$. We first give the computations of the remainder of the limits where we write computations for the transpose in the same line. 
\[ 
\begin{array}{c} W_{3,12} \to W_{3,4} : \\
\begin{pmatrix}  1& 0 & n \\
0 & 1 & 0\\
 0 & 0 & 1
\end{pmatrix} 
\end{array} 
\qquad 
\begin{array}{c} 
W_{3,10/11} \to W_{3,4}: \\
\begin{pmatrix} 
n& 0 & 0 \\
0 & \frac{1}{n} & 0 \\
0 & 0 & 1
\end{pmatrix} 
\end{array} 
\qquad 
\begin{array}{c} 
 W_{3,8/9}^z \to W_{3,4}: \\
\begin{pmatrix} 
n& 0 & 0 \\
0 & \frac{1}{n} & 0 \\
0 & 0 & 1
\end{pmatrix} 
\end{array} 
\qquad 
\begin{array}{c} 
W_{3,5/6}^{( a , b)} \to W_{3,4}: \\
 \begin{pmatrix}  1 & 0 & 0 \\
 0 & 1 & n \\
  0 & 0 & 1
\end{pmatrix} 
\end{array} 
\] 
the last limit includes $W_{3,2}$ and $W_{3,3}$ as singular cases, $W_{3,2/3} =W_{3,5/6}^{( a=b)} \to W_{3,4}$.  Finally, 
$$\begin{array}{c} W_{3,7} \to W_{3,5/6}^{b=0}\\  \begin{pmatrix}  1 & 0 & 0 \\
 0 & n & 0 \\
  0 & 0 & \frac{1}{n}
\end{pmatrix} 
\end{array} 
\qquad 
\begin{array}{c}  W_{3,1} \to W_{3,5/6} ^{(a=b)} \\ \begin{pmatrix} 1 & n & 0 \\
0 & 1 & 0 \\
 0 & 0 & 1
\end{pmatrix}. 
\end{array} 
$$

Now it remains to prove nonexistence of the missing arrows. 
By Theorem \ref{algebraic} limits of $W_{3,7}$ are algebraic.  %We know $W_{3,7}$ cannot limit to any 3 dimensional algebraic group with 4 dimensional normalizer because of Theorem \ref{normalizer},  and 
Proposition \ref{char} rules out the rest of the options for limits of $W_{3,7}$ except for $W_{3,5/6}$. 

\begin{proposition} The only possible values of $(a,b)$ for the groups $W_{3,5/6}^{a,b}$ which can appear as limits of $W_{3,7}$ are conjugate to $W_{3,5/6}$ with $b=0$. 
\end{proposition}

\begin{proof}  Using the standard Iwasawa decomposition argument, any limit of $W_{3,7}$ is conjugate to a limit under an element of the Borel, $B$.  Conjugating by an element of $B$ leaves the diagonal invariant. 
%we see 
%\[ \begin{pmatrix} e^t & * & * \\
%0 & 1 & * \\
%0 & 0 & e^{-t} 
%\end{pmatrix}.
%\] 
Thus the limit must either be unipotent, or is conjugate to an upper triangular group with diagonal $\langle e^t, 1, e^{-t} \rangle$.  In the case $W_{3,5}$ the possibilities for the limit group are $a+b=0$ or $b=0$ which are conjugate by a permutation.  For $W_{3,6}$ we see $a=0$ and $b=0$ are conjugate. 
\end{proof} 

Next $W_{3,1}$ is algebraic, and so its limits must be algebraic, and by Proposition \ref{char} its limits must have real weights. Thus $W_{3,12}$ and $W_{3,8/9}$ cannot be limits because they have complex wieghts.  So the only possible limits are the algebraic groups in $W_{3,5/6} ^{(a,b)}$ where $a, b \in \mathbb{Q}$.  %Leitner \cite{Leitnersl3} computes limits of $W_{3,1}$ under ANY possible sequence of conjugating matrices, so we know the only possibility is a limit to $W_{3,5/6} ^{(a=b)}$. 
By Proposition \ref{char} the only possibilities are the singular $W_{3,5/6} ^{(a=b)}$.

%Further, notice if a group has a 2 dimensional abelian subgroup (the diagonal in $W_{3,1}$), then the limit must also have a 2 dimensional abelian subgroup.  This happens to $W_{3,5}$ and $W_{3,6}$ only in singular cases. 

%All groups in left column are algebraic, so their limits must be algebraic.  So we have the special cases in the middle column, where the groups are algebraic.  The characteristic polynomial and eigenvalue conditions provide the rest of the invariants. 

\begin{proposition} The group $W_{3,1}$ limits only to singular groups among $W_{3, 5/6}^{(a,b)}$. 
\end{proposition} 

\begin{proof} 
The subgroup $W_{3,1}$ is contained in the Borel $B$. 
Since $\SL_3( \RR)  = \SO(3) B$ and $\SO(3)$ is compact, it suffices to consider conjugating only by sequences of elements $g_n\in B$.  Notice
$B = W_{3,1} N'$ where $W_{3,1}$ contains $A$, the subgroup of diagonal matrices, and $N' = \{I+tE_{1,2}+sE_{2,3}|t,s\in\mathbb{R}\}$ (notice $N'$ is not a subgroup).  
So, it suffices to consider $g_n = I+t_n E_{1,2}+s_n E_{2,3} \in N'$, and to assume that such a sequence is unbounded, i.e  $t_n\to\infty$ or $s_n\to \infty$. 
Now, if \[ h_n = \begin{pmatrix} a_n & 0 & d_n \\
0 & b_n & 0 \\
 0 & 0 & c_n
\end{pmatrix}\qquad  \textrm{with } a_n b_n c_n =1  \] 
is a sequence of elements in $W_{3,1}$ so that $h_n^{g_n}$ converges then the entries
\[ (h_n^{g_n})_{1,2} = s_n(a_n - b_n),\; (h_n^{g_n})_{2,3} = t_n(b_n - c_n) \]
converge. Thus either $a_n - b_n \to 0$ or $b_n - c_n\to 0$ depending on $s_n\to \infty$ or $t_n\to\infty$. 
This shows that on the diagonal of the limit two of the entries are the same. 
The only possible limits which are algebraic, conjugate into $B$ and have two equal entries on the diagonal are $W_{3,5/6}^{(a=b)}$. Indeed, %$W_{3,5}^{(a=b)}$ is a limit of $W_{3,1}^{g_n}$ where 
$$ \begin{array}{c} W_{3,1} \to W_{3,5} ^{a=b} \\ \begin{pmatrix} 1 & n & 0 \\
0 & 1 & 0 \\
 0 & 0 & 1
\end{pmatrix} \end{array}  
\qquad 
\begin{array}{c} W_{3,1} \to W_{3,6} ^{a=b} \\
 \begin{pmatrix} 1 & 0 & 0 \\
0 & 1 & n \\
 0 & 0 & 1
\end{pmatrix}. 
\end{array}   $$
\end{proof} 

It is not possible for $W_{3,8/9} ^z$ with $z \neq i$ to be a limit of another group, since $W_{3,8/9} ^z$ for $z \neq i$ is not algebraic and $W_{3,13}= \SO(2,1)$ and $W_{3,14} =\SO(3)$ are algebraic.    Also $W_{3,8/9} ^z$ cannot be a limit of $W_{3,1}$ or $ W_{3,7}$ by Proposition \ref{char}. 
Finally, it remains to check that $W_{3,4}$ is the only possible limit of $W_{3,8/9} ^z$.   
% These groups are not algebraic, so we first show limits have the same dimension. 
% %
%\begin{lemma}
%	Limits of  $W_{3,8/9} ^z$ are 3 dimensional. \todo{moved to second section. delete this lemma.}
%\end{lemma} 
%\begin{proof} 
%	Again applying the Iwasawa decomposition argument, it suffices to conjugate by elements in $B$. Since $H=W_{3,8}$ and $B$ are contained in the parabolic subgroup $Q = W_{6,1}$. 
%	%$$ Q = \begin{pmatrix} 
%	%* & *  & * \\
%	%* & *  & * \\
%	%0 & 0 & *
%	%\end{pmatrix}$$
%	It suffices to look at $\Sub(Q)$. 
%	Let $p: Q \to \GL_2 (\RR)$ be the homomorphism sending a matrix in $Q$ to its upper left $2\times2$ block. This homomorphism induces a homeomorphism $p^*: \Sub(\GL_2(\RR)) \to \{ H \in \Sub(Q) | H \ge \ker p\}$.
%	The image of $H$ is the 1-dimensional subgroup $\bar{H}=(e^{zt}) \in \Sub(Q)$. %By a similar argument to the group $W_{1,2}^z$ we see that a limit $\bar{L}$ of conjugates of $H$ must be 1-dimensional.  
%	Since $\bar H$ is acts uniformly properly, limits are 1 dimensional. 
%	Let $L$ be a limit of conjugates of $H$, then $L\le p^*(\bar{L})$ for some limit  $\bar{L}$ of conjugates of $\bar{H}$. Since $\ker p$ is 2-dimensional, $p^*(\bar{L})$ is 3-dimensional, and thus $L$ is 3-dimensional.\end{proof} 
%	  %
	 Since Proposition \ref{char} implies limits of $W_{3,8/9} ^z$ are unipotent, the only possibility is the limit we computed above to $W_{3,4}$. 

\subsection*{Dimension 4}

$$\begin{array}{c|c|c|c} \textrm{Name} & \textrm{Group} & \textrm{Normalizer} %& \dim (\textrm{Normalizer}) 
&\textrm{Properties} \\  
\hline 
W_{4,1} &  \begin{pmatrix} GL_2 & 0 \\ 0 & \editA{\det^{-1} \nolimits} \end{pmatrix} &W_{4,1} & %4 &
 \begin{array} {c} \cong \GL_2 ( \RR) \\ \textrm{algebraic} \end{array}  \\ 
\hline 
W_{4,2} & \begin{pmatrix} * & * & * \\ 0 & * & 0 \\ 0 & 0 & * \end{pmatrix}  & W_{4,2} & %4 &
 \begin{array}{c}  \cong \Aff(\mathbb{R}^2 )\\ \textrm{algebraic} \end{array}  \\
\hline 
W_{4,3} & \begin{pmatrix} * & 0 & * \\ 0 & * & * \\ 0 & 0 & * \end{pmatrix}  & W_{4,3}  &% 4 &
\begin{array}{c}  \cong \Aff(\mathbb{R}^2 )\\ \textrm{algebraic} \end{array} \\
\hline 
W_{4,4} & \begin{pmatrix} \mathbb{C}^* & \mathbb{C} \\ 0 & \editA{\det^{-1} \nolimits} \end{pmatrix} & W_{4,4}  & %4 &
 \begin{array}{c}  \cong \mathbb{C}^*\ltimes \mathbb{C} \\ \textrm{algebraic} \end{array}\\ 
\hline 
W_{4,5} & \begin{pmatrix} \editA{\det^{-1} \nolimits}& \mathbb{C} \\ 0 & \mathbb{C}^* \end{pmatrix} &   W_{4,5} & %4 & 
 \begin{array}{c}  \cong \mathbb{C}^*\ltimes \mathbb{C} \\ \textrm{algebraic} \end{array} \\ 
\hline 
W_{4,6}^{(a,b)} & \begin{pmatrix} e^{at} & * & * \\ 0 & e^{bt} & * \\ 0 & 0 & e^{-(a+b)t} \end{pmatrix}  &
%\begin{pmatrix} * & * & * \\
%0 & * & * \\
%0 & 0 & *
%\end{pmatrix}  
W_{5,3} &% 5 &
\begin{array}{c}  \cong \Heis(\mathbb{R}) \rtimes W_{1,1} \\ a= b \textrm{ algebraic} \\ a \neq b \textrm{ definable} \end{array}
\end{array} $$ 

\iffalse 
The first group is isomorphic to $\GL(2, \RR)$. 
%
\[
W_{4,1} = \begin{pmatrix} GL_2 & 0 \\ 0 & \det^{-1} \end{pmatrix} \qquad W_{4,1} 
\]
%
The next two groups are isomorphic to $\Aff(\mathbb{R})^2$. 
%
\[
W_{4,2} = \begin{pmatrix} * & * & * \\ 0 & * & 0 \\ 0 & 0 & * \end{pmatrix}  \qquad W_{4,2} 
\qquad \qquad 
W_{4,3} = \begin{pmatrix} * & 0 & * \\ 0 & * & * \\ 0 & 0 & * \end{pmatrix}  \qquad W_{4,3} 
\]
%
The next two groups are isomorphic to $\mathbb{C}^*\ltimes \mathbb{C}$. 
%
\[
W_{4,4} = \begin{pmatrix} \mathbb{C}^* & \mathbb{C} \\ 0 & \det^{-1} \end{pmatrix} \qquad W_{4,4} 
\qquad \qquad 
W_{4,5} = \begin{pmatrix} \det^{-1} & \mathbb{C} \\ 0 & \mathbb{C}^* \end{pmatrix} \qquad  W_{4,5} 
\]
%
The final group is isomorphic to $\Heis(\mathbb{R}) \rtimes W_{1,1}$. 
%
\[
W_{4,6}^{(a,b)} = \begin{pmatrix} e^{at} & * & * \\ 0 & e^{bt} & * \\ 0 & 0 & e^{-(a+b)t} \end{pmatrix}  \qquad 
\begin{pmatrix} * & * & * \\
0 & * & * \\
0 & 0 & *
\end{pmatrix} 
\]
%
%
%\subsection{Limits} 
%Notice the only case of $W_{4,6}^{(a,b)}$ which is algebraic is the singular case $a=b$.  
%
\fi 
The chart of local limits is as follows. 
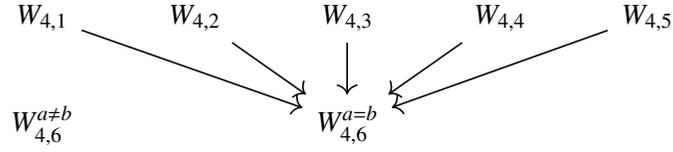
\begin{center} 
\begin{tikzcd} 
W_{4,1} \arrow[drr] & W_{4,2} \arrow[dr] & W_{4,3} \arrow[d] & W_{4,4} \arrow[dl] & W_{4,5} \arrow[dll] \\
W_{4,6} ^{a\neq b} && W_{4,6} ^{a=b} && 
\end{tikzcd} 
\end{center} 
%
%\begin{center} 
%\begin{turn}{90}
%\begin{tikzcd} 
%W_{4,1} \arrow[rdd] &  W_{4,6} ^{a\neq b} \\
%W_{4,2} \arrow[rd] % \arrow[ru, "?"] 
%&  \\
%W_{4,3} \arrow[r]  %\arrow[ruu, "?"]
%& W_{4,6} ^{a=b}\\
%W_{4,4} \arrow[ru] &  \\
%W_{4,5} \arrow[ruu] 
%\end{tikzcd} 
%\end{turn}
%\end{center}
%
%
%
%Notice the only case of $W_{4,6}^{(a,b)}$ which is algebraic is when $a, b \in \mathbb{Q}$. 
%
\[ \begin{array}{c} 
W_{4,1} \to W_{4,6} ^{a=b}:  \\
\begin{pmatrix} 1 & 0 & n \\
 0 & 1 & 0 \\
  0 & 0 &1 
\end{pmatrix} 
\end{array} 
\qquad 
\begin{array}{c} 
W_{4,2} \to  W_{4,6} ^{a=b}:  \\
\begin{pmatrix} 1 & 0 & 0 \\
 0 & 1 & n \\
  0 & 0 &1 
\end{pmatrix} 
\end{array} 
\qquad 
\begin{array}{c} 
W_{4,3} \to  W_{4,6} ^{a=b}: \\
\begin{pmatrix} 1 & n & 0 \\
 0 & 1 & 0 \\
  0 & 0 &1 
\end{pmatrix} 
\end{array} 
\qquad 
 \begin{array}{c}
W_{4,4} \to  W_{4,6} ^{a=b}: \\
\begin{pmatrix} n & 0 & 0 \\
 0 & \frac{1}{n} & 0 \\
  0 & 0 &1 
\end{pmatrix} 
\end{array} 
\qquad 
\begin{array}{c} 
W_{4,5} \to  W_{4,6} ^{a=b}: \\
\begin{pmatrix} 1 & 0 & 0 \\
 0 & n & 0 \\
  0 & 0 &\frac{1}{n} 
\end{pmatrix} 
\end{array} 
\] 
%
%We thus have the following graph of limits (not connected) 
%
%
%
By Theorem \ref{algebraic}  $W_{4,6}^{(a,b)}$ is the only possible limit of the first 5 groups, and limits of the first 5 groups are algebraic. Notice $W_{4,6}^{(a,b)}$ is algebraic for $a,b \in \mathbb{Q}$, but we claim only the singular groups $W_{4,6}^{(a=b)}$ are possible as limits of the first 5 groups.  Apply the Iwasawa decomposition $G= KNA$, and notice we only need to conjugate by elements of $B=NA$, since $K$ is compact then conjugating by elements of $K$ will not change the limit.

To show that $W_{4,3}\not \to W_{4,6}^{(a\ne b)}$, we note that $B=W_{4,3} U$, where $U= \{ Id +t \cdot E_{1,2} \}$. So it suffices to consider conjugating by sequences in $U$. But this is what we computed in the limits of $W_{4, 2/3}$ above. Similarly $W_{4,2}\not \to W_{4,6}^{(a\ne b)}$.

To show that $W_{4,1} \not \to W_{4,6}^{a\ne b}$.
 Using the standard Iwasawa decomposition argument, it is easy to verify that $G=W_{4,1} N' \SO(3)$ where $N' = \left( \begin{smallmatrix}
	1 & 0 &  * \\ 0 & 1 & 0 \\ 0 & 0 & 1
	\end{smallmatrix}\right)$. Since $\SO(3)$ is compact, it suffices to check for limits of conjugates of $W_{4,1}$ by elements of $N'$. However, this is exactly the limit computed above.

%
% {recall it is only necessary to check conjugation by elements of the Borel group.  But $ W_{4,1} B = W_{4,1} N$ where $N = I + n E_{13} + m E_{23}$. {\purple We need to show left multiplication by $N$ does not move the group.  Here we consider a compact group $\SO(2)$ direct product with $1$ in the last one by one block.  Thus we must have $m=0$, and for similar compactness reasons we can restrict to looking only inside this block, so we have already computed the limit above.}
%
%{\blue But $BW_{4,1} \subseteq  SO(3) N W_{4,1} $ where $N=I+nE_{1,3}$. Thus, the conjugating elements $p_n$ can in fact be chosen in $N$. However, this is the sequence for which we already computed a limit above.}
%
%
%{\purple  Notice $BW_{4,1} \subseteq  \left( \begin{smallmatrix} SO(2) & 0 \\ 0 & 1 \end{smallmatrix} \right)   N W_{4,1} $ where $N=I+nE_{1,3}$.  But the action by  $\left( \begin{smallmatrix} SO(2) & 0 \\ 0 & 1 \end{smallmatrix} \right)$ can be disregarded for compactness reasons, as above.  So we only need to understand limits under conjugation by $N$, which is what we have computed above.  }

%Any element in $N$ can be left multiplied by an element of $W_{4,1}$ to get an element of $N$ with $n=0$.  This is the sequence for which we already computed a limit above.   

Finally $W_{4,4}, W_{4,5} \not \to W_{4,6} ^{a\neq b}$ by Proposition \ref{char}, since if $W_{4,4}, W_{4,5}$ have real weights then two of them must be equal. 

\subsection*{Dimension 5}
Recall a subgroup is parabolic if and only if it contains the Borel subgroup. In our case, a subgroup of $G=\SL_3(\RR)$ is parabolic if it contains $B=W_{5,3}$.  
%The space $G/B$ is compact and hence $G/N_G(P)$ is compact for every parabolic subgroup $P\le G$. 
%The subspace of conjugates of a parabolic subgroup $P\le G$ is homeomorphic to $G/N_G(P)$ and hence it is compact in $\Sub(G)$, and in particular closed.
%That is, parabolic subgroups cannot locally converge to non-conjugate subgroups.
%Since all subgroup of dimension 5 are parabolic, the chart of local limits consists of isolated points.
%
\editA{The space $G/B$ is compact and hence $G/N_G(H)$ is compact for every subgroup $H\le G$ with a parabolic normaliser $P=N_G(H)$. For such a subgroup, the subspace of its conjugates is homeomorphic to $G/N_G(H)$ and hence is compact in $\Sub(G)$, and in particular closed. Thus subgroups with parabolic normalizers cannot locally converge to non-conjugate subgroups. Since all subgroups of dimension 5 are have parabolic normalisers, the chart of local limits consists of isolated points.}

%the space of conjugates of any parabolic subgroup, $G/N_G(P)$ is closed, so the chart of local limits is three isolated points. 

$$\begin{array}{c|c|c|c} \textrm{Name} & \textrm{Group} & \textrm{Normalizer} %& \dim (\textrm{Normalizer}) 
&\textrm{Properties} \\  
\hline 
W_{5,1} &  \begin{pmatrix} SL_2 & \begin{matrix} * \\ * \end{matrix}  \\ \begin{matrix} 0 & 0\end{matrix} & 1 \end{pmatrix} & 
%\begin{pmatrix} * & * & * \\
%* & * & * \\
%0 & 0 & *
%\end{pmatrix}  
W_{6,1} & %6 & 
\begin{array}{c}  \cong \SL_2(\mathbb{R}) \ltimes \mathbb{R}^2 \\ \textrm{algebraic} \end{array} \\
\hline 
W_{5,2} &  \begin{pmatrix} 1 & \begin{matrix} * & * \end{matrix}  \\ \begin{matrix} 0 \\ 0 \end{matrix} & SL_2 \end{pmatrix} & 
%\begin{pmatrix} * & * & * \\
%0 & * & * \\
%0 & * & *
%\end{pmatrix}  
W_{6,2} & %6 & 
 \begin{array}{c}  \cong \SL_2(\mathbb{R}) \ltimes \mathbb{R}^2 \\ \textrm{algebraic} \end{array} \\
\hline 
W_{5,3} & \begin{pmatrix} * & * & * \\ 0 & * & * \\ 0 & 0 & * \end{pmatrix} & W_{5,3} &% 5 & 
\begin{array}{c}  \cong \textrm{Borel} \\ \textrm{algebraic} \end{array} \\
\end{array} $$ 
%
\iffalse 
The first two groups are isomorphic to $\SL_2(\mathbb{R}) \ltimes \mathbb{R}^2$. 
%
\[
W_{5,1} = \begin{pmatrix} SL_2 & \begin{matrix} * \\ * \end{matrix}  \\ \begin{matrix} 0 & 0\end{matrix} & 1 \end{pmatrix}\qquad 
\begin{pmatrix} * & * & * \\
* & * & * \\
0 & 0 & *
\end{pmatrix} 
\qquad \qquad 
W_{5,2} = \begin{pmatrix} 1 & \begin{matrix} * & * \end{matrix}  \\ \begin{matrix} 0 \\ 0 \end{matrix} & SL_2 \end{pmatrix}\qquad 
\begin{pmatrix} * & * & * \\
0 & * & * \\
0 & * & *
\end{pmatrix}  
\]
%
The third subgroup $W_{5,3}$ is the Borel subgroup which is self normalizing.
\fi 
%
%\%[
%W_{5,3} = \begin{pmatrix} * & * & *  \\ 0 & * & * \\ 0 & 0 & * \end{pmatrix} \qquad \textrm{self-normalizing} 
%\]
%

%Notice the space of conjugates of $W_{5,3}$ is closed because it is a parabolic subgroup.   This is true for all parabolic groups.  We computed $G /B$ in the handwritten notes. 

\subsection*{Dimension 6}%%%%%%%%%%%%%%%%%%%%%
The subspace of conjugates of a parabolic subgroup is closed in the Chabauty compactification.  So the chart of limits is two isolated points. 
$$\begin{array}{c|c|c|c} \textrm{Name} & \textrm{Group} & \textrm{Normalizer} %& \dim (\textrm{Normalizer}) 
&\textrm{Properties} \\  
\hline 
W_{6,1} &  \begin{pmatrix} GL_2 & *  \\ 0 &\editA{\det^{-1} \nolimits} \end{pmatrix} & W_{6,1} & %6 &
 \textrm{algebraic} \\ 
\hline 
W_{6,2} &  \begin{pmatrix} \editA{\det^{-1} \nolimits} & *  \\ 0 & GL_2 \end{pmatrix}& W_{6,2} & %6 &
 \textrm{algebraic} \\ 
\end{array} $$ 
\iffalse 
Both groups are isomorphic to $GL_2(\mathbb{R})\ltimes \mathbb{R}^2$ (Parabolics) and both are self-normalizing. 
%
\[
W_{6,1} = \begin{pmatrix} GL_2 & *  \\ 0 & \det^{-1} \end{pmatrix} \qquad \qquad 
W_{6,2} = \begin{pmatrix} \det^{-1} & *  \\ 0 & GL_2 \end{pmatrix}
\]
\fi
%

\subsection*{Dimensions 7 and 8}

Only $\SL_3(\RR)$ is of dimension 8.   There are no subgroups of dimension 7.

\end{document}